\documentclass[12pt]{amsart}
\scrollmode

\usepackage{amsmath,amsthm,amsfonts,fullpage}
\usepackage[all]{xy}
\newtheorem{thm}{Theorem}
\newtheorem{prop}{Proposition}
\newtheorem{lem}{Lemma}
\newtheorem{cor}{Corollary}
\newtheorem{conj}{Conjecture}
\newtheorem*{thmasphericity}{Theorem \ref{asphericity}}
\newtheorem*{thmLefschetzbundles}{Theorem \ref{Lefschetzbundles}}
\usepackage{amssymb}

\theoremstyle{remark}
\newtheorem{rem}{Remark}
\newtheorem{ex}{Example}

\theoremstyle{definition}
\newtheorem{defn}[thm]{Definition}

\title{The topology of symplectic circle bundles}
\author{Jonathan Bowden}
\address{Mathematisches Institut, Ludwig-Maximilians-Universit\"at, Theresienstr. 39, 80333 M\"unchen, Germany}
\email{jonathan.bowden@mathematik.uni-muenchen.de}
\date{19 Nov 2007. 2000 $Mathematics \ Subject \ Classification$. Primary 57R17; Secondary 57N10, 57N13}

\begin{document}

\begin{abstract}
We consider circle bundles over compact three-manifolds with symplectic total spaces. We show that the base of such a space must be irreducible or the product of the  two-sphere with the circle. We then deduce that such a bundle admits a symplectic form if and only if it admits one that is invariant under the circle action in three special cases: namely if the base is Seifert fibered, has vanishing Thurston norm, or if the total space admits a Lefschetz fibration.
\end{abstract}

\maketitle

\section{Introduction}
A conjecture due to Taubes states that if a closed, compact 4-manifold of the form $M \times S^1$ is symplectic, then $M$ must fiber over $S^1$. A natural extension of this conjecture is to the case where  $E \stackrel{\pi} \rightarrow M$ is a possibly nontrivial circle bundle. In \cite{FGM} it was shown that if an $S^1$-bundle admits an $S^1$-invariant symplectic form then the base must fiber over $S^1$ and the Euler class $e(E)$ pairs trivially with the fiber of some fibration. Thus based on the principle that an $S^1$-bundle should admit a symplectic form if and only if it admits an invariant one, one arrives at the following conjecture.
\begin{conj}[Taubes]\label{strong}
If a circle bundle $S^1 \to E \stackrel{\pi} \rightarrow M$ over a closed, compact 3-manifold is symplectic, then there is a fibration $\Sigma \to M \stackrel{\phi}\rightarrow  S^1$ such that $e(E)([\Sigma]) = 0$.
\end{conj}
If an oriented 3-manifold fibers over $S^1$ with fiber $\Sigma \neq S^2$, then it follows by the long exact homotopy sequence that $M$ is in fact aspherical. So a necessary condition for Conjecture \ref{strong} to hold is that any $M$ that is the base of an $S^1$-bundle, whose total space carries a symplectic form, must in fact be aspherical or $S^2 \times S^1$ in the case $\Sigma = S^2$. This observation provides the motivation for the following theorem, which is the main result of the first part of this paper.
\begin{thmasphericity}
Let $M$ be an oriented, closed 3-manifold, so that some circle bundle $S^1 \rightarrow E \stackrel{\pi} \rightarrow M$ admits a symplectic structure, then, either $M$ is diffeomorphic to $S^2 \times S^1$ and the bundle is trivial, or $M$ is irreducible and aspherical.
\end{thmasphericity}
A similar statement was proved by McCarthy in \cite{Mc} for the case $E = M \times S^1$. More precisely, McCarthy showed that if $M \times S^1$ admits a symplectic structure then $M$ decomposes as a connected sum $M=A \# B$ where the first Betti number $b_1(A) \geq 1$ and $B$ has no nontrivial connected covering spaces. This can be refined quite substantially following Perelman's proof of Thurston's geometrisation conjecture (see \cite{Per1}, \cite{Per2} or \cite{MoT}). For one corollary of geometrisation is that the fundamental group of a closed 3-manifold is residually finite (see \cite{H}), meaning that the $B$ in McCarthy's theorem must have trivial fundamental group, and hence by the Poincar\'e Conjecture is diffeomorphic to $S^3$. Thus in fact $M$ must be prime and hence irreducible and aspherical or $S^2 \times S^1$. Theorem \ref{asphericity} is then a generalisation of this more refined statement to the case of nontrivial $S^1$-bundles. Our argument will rely on a vanishing result of Kronheimer-Mrowka for the Seiberg-Witten invariants of a manifold that splits into two pieces along a copy of $S^2 \times S^1$, which in itself is of independent interest (cf. Proposition \ref{vanishingSW}). One may also prove Theorem \ref{asphericity} by following the argument of \cite{Mc}, see Remark \ref{McProof} below.

In the remainder of this paper we will show that Conjecture \ref{strong} holds in various special cases. Firstly we will verify the conjecture under certain additional assumptions on the topology of the base manifold $M$. In order to be able to do this we will need to understand when a manifold fibers over $S^1$. One gains significant insight into this problem by considering the Thurston norm $||\ ||_T$ on $H^1(M, \mathbb{R})$, which was introduced by Thurston in \cite{Th}. The Thurston norm enables one to see which integral classes $\alpha \in H^1(M, \mathbb{Z})$ can be represented by closed, nonvanishing 1-forms, which in turn induce fibrations of $M$ by compact surfaces.

In \cite{FVi} it was shown that if $E = M \times S^1$ admits a symplectic form and $||\ ||_T \equiv 0$ or $M$ is Seifert fibered, then $M$ must fiber over $S^1$. In Corollary \ref{Seifertcase} below we will show that in fact Conjecture \ref{strong} holds in these two cases. The argument will be based on understanding the Seiberg-Witten invariants of the total space $E$ given that $M$ has vanishing Thurston norm and the Seifert case will be deduced as a corollary of this. Indeed, if $M$ has vanishing Thurston norm and $S^1 \to E \stackrel{\pi} \rightarrow M$ is symplectic, then the canonical class of $E$ must be trivial. This combined with the restrictions on Seiberg-Witten basic classes of a symplectic manifold as proved by Taubes in \cite{Tau2} means that $K = 0$ is the only Seiberg-Witten basic class and the result then follows by an application of a vanishing result of Lescop (cf. \cite{Les} or \cite{Tur}).

Another special case of the Taubes conjecture is when the total space $E$ admits a Lefschetz fibration, as was considered in \cite{CM} and \cite{Etg} for a trivial bundle. In view of Corollary \ref{Seifertcase} we will be able to give a comparatively simple proof of the following result.
\begin{thmLefschetzbundles}
Let $S^1 \to E \stackrel{\pi} \rightarrow M$ be a symplectic circle bundle over an irreducible base $M$. If $E$ admits a Lefschetz fibration, then $M$ fibers over $S^1$.
\end{thmLefschetzbundles}
It then follows by considering the Kodaira classification of complex surfaces that Conjecture \ref{strong} holds under the assumption that the total space admits a complex structure.

$Outline \ of \ paper$. In Section \ref{Asphericalbase} we will state the relevant vanishing result of Kronheimer-Mrowka in order to prove Theorem \ref{asphericity}. In Section \ref{Thurstonnorm} we recall the definition of the Thurston norm and quote some well known facts about it. In Section \ref{VanishingThNorm} we will use our knowledge of the Thurston norm to verify Conjecture \ref{strong} under the assumption that the base is Seifert fibered or has vanishing Thurston norm. Finally in Section \ref{The case where E admits a Lefschetz fibration} we will define Lefschetz fibrations and prove that the conjecture is true when one has a Lefschetz fibration on the total space $E$.

$Acknowledgments$.  I would like to thank Professor Dieter Kotschick for his wise and patient supervision that culminated in this paper.

\section{Asphericity of the base $M$}\label{Asphericalbase}
Throughout this article all manifolds will be closed, connected, compact and oriented and $M$ will always denote a manifold of dimension 3. In addition we will make the convention that all (co)homology groups will be taken with integral coefficients unless otherwise stated. 

In \cite{Mc} it was shown that if $M \times S^1$ is symplectic, then $M$ must be irreducible and aspherical or $S^2 \times S^1$. We extend this to the case of a nontrivial $S^1$-bundle. We first collect some relevant lemmas.
\begin{lem} \label{coveringlemma}
Let $M =M_1 \# M_2$ be a nontrivial connect sum decomposition with $b_1(M) \geq 1$, then there is a finite covering $N$ of $M$ that decomposes as a direct sum $N =N_1 \# N_2$ where $b_1(N_i) \geq k$ for any given $k$.
\end{lem}
\begin{proof} It follows from Mayer-Vietoris that the Betti numbers are additive for a connect sum, hence by assumption we may assume that $b_1(M_1) \geq 1$. By the proof of geometrisation it follows that the fundamental group of a 3-manifold is residually finite (cf. \cite{H}) and hence $M_2$ has a nontrivial $d$-fold cover $\tilde{M}_2$, with $d \geq 2$. By removing a ball from $M_2$ and its disjoint lifts from $\tilde{M}_2$ and then gluing in $d$ copies of $M_1$ we obtain a cover $\tilde{M}$ of $M =M_1 \# M_2$, and by construction $\tilde{M}$ has a connect sum decomposition as $\tilde{M} = M_1 \# P$, where $b_1(P) \geq 1$. We may now take a $k$-fold cover associated to some surjective homomorphism of $\pi_1(M_1) \rightarrow \mathbb{Z}_k$ and glue in copies of $P$ to get a cover of $\tilde{M}$ (and hence of $M$), which decomposes in two pieces one of which has first Betti number at least $k$. One more application of this procedure gives the desired result.
\end{proof}

\begin{lem} \label{gysinlemma}
Let $S^1 \to E \stackrel{\pi} \rightarrow M$ be a circle bundle, whose Euler class we denote by $e(E) \in H^2(M)$, then 
\begin{enumerate}
\item $b_2(E) = \begin{cases} 
   2b_1(M) - 2,  & \mbox{if }e(E)\mbox{ is not torsion} \\
   2b_1(M), & \mbox{if }e(E)\mbox{  is torsion} 
\end{cases}$

\item $b_2^{+}(E) = b_2^{-}(E) \geq b_1(M) - 1$.
\end{enumerate}
\end{lem}
\begin{proof}
We consider the Gysin sequence
\begin{equation*}
H^0(M) \stackrel{\cup e} \rightarrow H^2(M) \stackrel{\pi^{*}} \rightarrow H^2(E) \stackrel{\pi_{*}} \rightarrow H^1(M) \stackrel{\cup e} \rightarrow H^3(M),
\end{equation*}
where here $e \in H^2(M)$ denotes the Euler class of the bundle. By Poincar\'e duality $H^0(M) = H^3(M) = \mathbb{Z}$ and $b_1(M) = b_2(M)$, so we conclude by exactness that $b_2(E) =  2b_1(M) - 2$ if $e$ is not torsion and $b_2(E) = 2b_1(M)$ if $e$ is torsion. Furthermore since $E$ bounds its associated disc bundle, it has zero signature and hence
\[ b_2^{+}(E) = b_2^{-}(E) \geq b_1(M) - 1.\]
\end{proof}

We will need to appeal to a vanishing result for the Seiberg-Witten invariants of manifolds that decompose along $S^2 \times S^1$, which we take from \cite{KM}. For this we will need to define a relative notion of $b_2^{+}$ for an oriented 4-manifold $X$ with boundary. This is done by considering the symmetric form induced on rational cohomology that is obtained as the composition
\begin{equation*}
H^2(X, \partial{X})  \times H^2(X, \partial{X})  \stackrel{i^* \times Id} \rightarrow H^2(X) \times H^2(X, \partial{X}) \stackrel{\cup} \rightarrow \mathbb{Q} \ .
\end{equation*}
Here the map $i^*$ is the map coming from the long exact sequence of the pair $(X, \partial{X})$ and the second map is non-degenerate by Poincar\'e duality. This is then a symmetric, possibly degenerate, form on $H^2(X, \partial X)$  and we define $b_2^{+}(X)$ to be the dimension of a maximal positive definite subspace.
\begin{thm}[Kronheimer-Mrowka, \cite{KM}]\label{KronMro}
Let $X = X_1 \bigcup_{\partial{X_1} = \partial{X_2}} X_2$ where $\partial{X_1} = - \partial{X_2} = S^2 \times S^1$ and $b_2^{+}(X_1), b_2^{+}(X_2) \geq 1$. Then for all $Spin^c$-structures $\xi$
\[ \sum_{\xi^* - \xi \in \ Tor} SW(\xi^*)=0.\]
\end{thm}
Although it is not explicitly stated in book \cite{KM}, Theorem \ref{KronMro} can be deduced as follows: formula $3.27$ (p.73) defines the sum of the SW invariants of all $Spin^c$- structures that differ by torsion as given by a pairing of certain Floer groups and these groups are zero for $S^2 \times S^1$ by Proposition 3.10.3 in the case of an untwisted coefficient system and by Proposition 3.10.4 in the twisted case.

Theorem \ref{KronMro} then implies certain restrictions on the decomposition of symplectic manifolds along $S^2 \times S^1$, which is similar but slightly weaker than the results one obtains for a connected sum.

\begin{prop}\label{vanishingSW}
A symplectic manifold $X$ cannot be decomposed as $X = X_1 \bigcup_{\partial{X_1} = \partial{X_2}} X_2$, where $\partial{X_1} = - \partial{X_2} = S^2 \times S^1$ and $b_2^{+}(X_1), b_2^{+}(X_2) \geq 1$.
\end{prop}
\begin{proof}
By the hypotheses of the proposition, we conclude from Theorem \ref{KronMro} that for every $Spin^c$-structure $\xi \in Spin^c(X)$
\[ \sum_{\xi^* - \xi \in \ Tor} SW(\xi^*)=0.\]
However as $X$ is symplectic and 
\[ b_2^+(X) \geq b_2^{+}(X_1) + b_2^{+}(X_2) \geq 2\]
the nonvanishing result of Taubes implies $SW(\xi_{can}) = \pm 1$, where $\xi_{can}$ denotes the canonical $Spin^c$-structure associated to the symplectic structure on $E$ (cf. \cite{Tau}). Moreover it follows from the constraints on SW basic classes of a symplectic manifold of \cite{Tau2} that if $\xi^*$ is another $Spin^c$-structure with non-trivial SW invariant and $\xi_{can} - \xi^* \in Tor$ then in fact $\xi_{can} = \xi^*$. Hence
\[ \sum_{\xi^* -  \xi_{can} \in  \ Tor} SW(\xi^*)= \pm 1\]
which is a contradiction.
\end{proof}

\begin{thm}\label{asphericity}
Let $M$ be an oriented, closed 3-manifold, so that some circle bundle $S^1 \rightarrow E \stackrel{\pi} \rightarrow M$ admits a symplectic structure, then $M$ is irreducible and aspherical or $M = S^2 \times S^1$ and the bundle is trivial.
\end{thm}
\begin{proof}
We first show that $M$ must be prime. Since $E$ is symplectic it follows from Lemma \ref{gysinlemma} that $b_1(M) \geq 1$. Assume that $M = M_1 \# M_2$ is a nontrivial connected sum, then by taking a suitable covering as in Lemma \ref{coveringlemma} and pulling back $E$ and its symplectic form we may assume without loss of generality that $b_1(M_i) \geq 2$. We let $S$ denote the gluing sphere of the connected sum, then as $S$ is nullhomologous the bundle restricted to $S$ is trivial. Thus the connect sum decomposition induces a decomposition $E = E_1 \bigcup_{S^2 \times S^1} E_2$. Since the bundles $E_i \to M_i/B^3$ are trivial on the boundary we may extend them to bundles $\tilde{E}_i \to M_i$ and as $b_1(M_i) \geq 2$,  Lemma \ref{gysinlemma} implies that $b_2^+(\tilde{E}_i) \geq 1$. Further, since $E_i \simeq \tilde{E}_i / S^1 \times pt$ we have that
\[b_2^+(E_i) \geq b_2^+(\tilde{E}_i) \geq 1,\]
which then contradicts Proposition \ref{vanishingSW}. Hence $M$ is prime, and thus irreducible or $S^2 \times S^1$. 

We assume that $M$ is irreducible, then by the sphere theorem $\pi_2(M) = 0$. Since $b_1(M) \geq 1$, we have that $\pi_1(M)$ is infinite so the universal cover $\tilde{M}$ of $M$ is not compact and has $\pi_i(\tilde{M})$ trivial for $i =1,2$. The Hurewicz theorem then implies that the first nontrivial $\pi_i(\tilde{M})$ is isomorphic to $H_i(\tilde{M})$. But since $\tilde{M}$ is not compact $H_3(\tilde{M}) = 0$ and as $\tilde{M}$ is 3-dimensional $H_i(\tilde{M}) = 0$ for all $i \geq 4$. Hence $\pi_i(\tilde{M}) = 0$ for all $i \geq 1$ and it follows from Whitehead's Theorem that $\tilde{M}$ is contractible, that is $M$ is aspherical.

In the case where $M = S^2 \times S^1$ any symplectic bundle must be trivial by Lemma \ref{gysinlemma}.
\end{proof}
\begin{rem}\label{McProof}
One may also give a proof of Theorem \ref{asphericity} that uses covering construction in \cite{Mc}. In order to do this one first takes finite coverings on each of the two pieces in the connect sum decomposition. Then one glues these together to find a covering $\tilde{M}$ where the sphere of the connect sum lifts to a sphere that is nontrivial in real cohomology. This sphere then lifts to the total space of the pullback bundle $\tilde{E}$ over $\tilde{M}$. One may also assume by Lemma \ref{coveringlemma} that $b_1(\tilde{M})$ is large and hence $b_2^+(\tilde{E})$ is large. Then a standard vanishing theorem for the SW invariants implies that all invariants are zero, which then contradicts Taubes' result if $E$ and hence $\tilde{E}$ is symplectic.
\end{rem}

By considering the long exact homotopy sequence we have the following corollary that was first proved by Kotschick in \cite{Kot2}.
\begin{cor}\label{nocontractibleorbits}
Let $S^1 \to E \stackrel{\pi} \rightarrow M$ be a symplectic circle bundle over an oriented 3-manifold $M$. Then the map $\pi_1(S^1) \to \pi_1(E)$ induced by the inclusion of the fiber is injective. In particular a fixed point free circle action on a symplectic 4-manifold can never have contractible orbits.
\end{cor}

\section{The Thurston norm}\label{Thurstonnorm}
In this section we will define and collect several relevant facts about the Thurston norm. We first define the negative Euler characteristic or \emph{complexity} of a possibly disconnected, orientable surface $\Sigma = \bigsqcup_i\Sigma_i$ to be
\begin{equation*}
\chi_{-}(\Sigma) =   \sum_{\chi(\Sigma_i) \leq 0} -\chi(\Sigma_i)
\end{equation*}
where $\chi$ denotes the Euler characteristic of the surface. 

Next we define the Thurston norm $||\ ||_T$ as a map on $H_1(M)$ by
\begin{equation*}
||\sigma||_T =  \min \{\chi_{-}(\Sigma) \ |\ PD(\Sigma) = \sigma \}.
\end{equation*}
It is a basic fact that this map extends uniquely to a (semi)norm on $H^1(M, \mathbb{R})$, which we will denote again by $||\ ||_T$. One particularly important property of the Thurston norm is that its unit ball, which we denote by $B_T$, is a (possibly noncompact) convex polytope with finitely many faces. If $B_{T^*}$ denotes the unit ball in the dual space we have the following characterisation of $B_T$.
\begin{thm}[\cite{Th}, p. 106]\label{unitball}
The unit ball $B_{T^*}$ is a polyhedron whose vertices are integral lattice points, $\pm\beta_1,..., \pm\beta_k$ and the unit ball $B_T$ is defined by the following inequalities
\begin{equation*}
B_T = \{\alpha \ | \ |\beta_i(\alpha)| \leq 1,\  1 \leq i \leq k\}.
\end{equation*}
\end{thm}
We are interested in understanding how a manifold fibers over $S^1$ and the following theorem says that the Thurston norm determines precisely which cohomology classes can be represented by fibrations.
\begin{thm}[\cite{Th}, p. 120]\label{fiberedfaces}
Let $M$ be a compact, oriented 3-manifold. The set $F$ of cohomology classes in $H^1(M, \mathbb{R})$ representable by nonsingular closed 1-forms is the union of the open cones on certain top-dimensional open faces of $B_T$, minus the origin. The set of elements in $H^1(M,\mathbb{Z})$ whose Poincar\'e dual is represented by the fiber of some fibration consists of the set of lattice points in $F$.
\end{thm}
We call a top-dimensional face of the unit ball $B_T$ \emph{fibered,} if some integral class, and hence all, in the cone over its interior can be represented by a fibration. One also understands how the Thurston norm behaves under finite covers by the following result of Gabai.
\begin{thm}[\cite{Gab}, Cor. 6.13]\label{immersedrep}
Let $\tilde{M} \to M$ be a finite connected $d$-sheeted covering then for $\sigma \in H^2(M, \mathbb{R})$ we have
\[||\sigma||_T =\frac{1}{d} ||p^*\sigma||_T.\]
\end{thm}

These facts then allow us to completely characterise the Thurston norm of an irreducible Seifert fibered manifold.
\begin{prop}\label{SFnorm}
Let $M$ be irreducible and Seifert fibered, then either the Thurston norm of $M$ vanishes identically or $M$ fibers over $S^1$ and 
\[||\sigma||_T = \chi(F)|\sigma(\gamma)|\]
where $\gamma \in H_1(M)$ is a primitive class some multiple of which is homologous to the fiber of a Seifert fibration and $F$ is a connected fiber of a fibration of $M$.      
\end{prop}
\begin{proof}
Since $M$ is irreducible and Seifert fibered either $M$ has a horizontal surface or every surface is isotopic to a vertical surface  (cf. \cite{Hat} Prop 1.11) and is hence a union of tori so the Thurston norm is identically zero. If $M$ has a horizontal surface $F$, which we may assume to be connected, then $M$ is a mapping torus with monodromy $\phi \in Diff^+(F)$ so that $\phi^n =Id$ for some $n$. This means that $M$ is covered by $\tilde{M} = F \times S^1$. If $\tilde{\gamma} = pt \times S^1$, then the Thurston norm of $\tilde{M}$ is given by
\[||\sigma ||_T = \chi(F)|\sigma(\tilde{\gamma})|\]
and the formula for the norm on $M$ follows from Theorem \ref{immersedrep}.
\end{proof}
\begin{ex}[Seifert fibered spaces with horizontal surfaces]\label{SFfiberedfaces}
We note that in the second case of Proposition \ref{SFnorm} the Thurston ball $B_T$ consists of two (noncompact) faces that are both fibered. Thus by \cite{FGM} any bundle over such an $M$ will admit an $S^1$-invariant symplectic form except possibly in the case where the Euler class $e(E)$ is a multiple of $PD(\gamma) \in H^1(M)$. If a bundle over $M$ with Euler class a multiple of $PD(\gamma)$ is symplectic then by taking the pullback bundle of the cover $\tilde{M} = F \times S^1 \to M$ we may assume that we have a bundle $E$ over $F \times S^1$ that is symplectic and has Euler class that is multiple of $PD(\tilde{\gamma})$. This in turn has a covering $\bar{E}$ with Euler class equal to $PD(\tilde{\gamma})$. Now if we let $T =\tilde{\gamma} \times S^1$ and  $X= \tilde{M} \times S^1$ then the SW polynomial of $X$ can be computed to be
\[ \mathcal{SW}_X^4 = (t_T - t_T^{-1})^{2g-2} \]
where $g$ is the genus of $F$. Then by the formula of Baldridge in \cite{Bal} it follows that all the SW invariants of $\bar{E}$ are zero, contradicting Taubes' non vanishing result. So in fact Conjecture \ref{strong} holds for Seifert fibered spaces that have horizontal surfaces.
\end{ex}
\section{The case of vanishing Thurston norm}\label{VanishingThNorm}
In \cite{FVi} it was shown that if $E = M \times S^1$ admits a symplectic form and $||\ ||_T \equiv 0$ or $M$ is Seifert fibered, then $M$ must fiber over $S^1$. In this section we shall extend this to the case of a nontrivial $S^1$-bundle and then show that Conjecture \ref{strong} holds in both of these cases.
From now on we shall assume that $M$ is also irreducible, which in view of Theorem \ref{asphericity} only excludes the case where $M = S^2 \times S^1$ and the bundle is trivial. Our argument will be based on that of \cite{FVi} and we begin with the following lemma.
\begin{lem}\label{tori}
If $S^1 \to E \stackrel{\pi} \rightarrow M$ is a bundle over an $M$ that has vanishing Thurston norm, then
\[ H^2(E, \mathbb{Z}) /Tor = V \oplus W \]
where $V,W$ are isotropic subspaces that admit a basis of embedded tori.
\end{lem}
\begin{proof}
We consider the Gysin sequence
\[ \xymatrix{ \mathbb{Z}  \ar[r] & H^2(M) \ar[r]^{ \pi^{*}} & H^2(E) \ar[r]^{\pi_{*}}& H^1(M) \ar@/  ^/ [l]^s \ar[r] & \mathbb{Z}.} \]
Here $s$ is a section defined on the image of $\pi_*$ as follows: we represent an element of $\sigma \in H^1(M)$ by an embedded surface $\Sigma$. By exactness, $\sigma$ will be in $Im(\pi_*)$ precisely when the bundle is trivial on $\Sigma$ and in this case we may lift $\Sigma$ to some  $\tilde{\Sigma}$ in $E$. As $H^1(M)$ is free, we define $s$ on a $\mathbb{Z}$-basis $\{\sigma_i\}$ by  $s(\sigma_i) = \tilde{\Sigma}_i$. We set $V = \pi^* H^2(M)$ and $W = s(H^1(M))$, then $V$ is clearly spanned by embedded tori and the statement for $W$ is precisely the assumption on the Thurston norm.
\end{proof}
\begin{prop}\label{killtorsion}
Let $S^1 \to E \stackrel{\pi} \rightarrow M$ be an $S^1$-bundle with torsion Euler class $e(E)$, then there is a finite cover $\tilde{M} \stackrel{p} \rightarrow M$ such that the pullback bundle $p^*E \to \tilde{M}$ is trivial.
\end{prop}
\begin{proof}
We choose a splitting of $H_1(M) = F \oplus T$ where $T$ is the torsion subgroup and $F$ is any free complement. We take the cover $\tilde{M} \stackrel{p} \rightarrow M$ associated to the kernel of the composition 
\[ \pi_1(M) \to H_1(M) \stackrel{\phi} \rightarrow T,\]
where $\phi$ is the projection with kernel $F$. Note that the composition $H_1(\tilde{M}) \stackrel{p_*} \rightarrow H_1(M) \stackrel{\phi_*} \rightarrow T$ is zero. Then by the Universal Coefficient Theorem we have the following commutative diagram
\[ \xymatrix{ 0 \ar[r]& Ext(H_1(\tilde{M}), \mathbb{Z}) \ar[r]& H^2(\tilde{M}) \ar[r]& Hom(H_2(\tilde{M}), \mathbb{Z}) \ar[r] & 0 \\
 0 \ar[r]& Ext(H_1(M), \mathbb{Z}) \ar[r] \ar[u]^{(p_*)^*}& H^2(M) \ar[r] \ar[u]^{p^*}& Hom(H_2(M), \mathbb{Z}) \ar[r] \ar[u]^{(p_*)^*}& 0. \\
 & \ar[u]^{(\phi_*)^*}_{\cong} Ext(T, \mathbb{Z})& &&}\]
This implies that $p^*$ is zero on torsion in $H^2(M)$ so the pullback bundle is indeed trivial.
 \end{proof}
 
 \begin{thm}\label{vanishingThnorm}
Let $S^1 \to E \stackrel{\pi} \rightarrow M$ be a symplectic circle bundle over an irreducible manifold for which $||\ ||_T$ is identically zero, then $M$ fibers over $S^1$. 
\end{thm}
\begin{proof}
Since $E$ is symplectic it has an associated canonical bundle $\xi_{can}$ and canonical class that we denote by $K$. We claim that our assumption on the Thurston norm of the base implies that $K$ must be torsion. For by Taubes' nonvanishing result $\xi_{can}$ has nontrivial SW invariant. If $\alpha \in H^2(E)$, the adjunction inequality (see \cite{Kro}) and Lemma \ref{tori} imply that
\[|\alpha . K| = 0.\]
This also holds in the case $b_2^+(E) = 1$ (cf. \cite{LL} Theorem E). As $M$ is irreducible and $b_2(M) \geq 1$ the assumption on the vanishing of the Thurston norm implies that $M$ contains an embedded, incompressible torus $T \hookrightarrow M$. Then by  Proposition 7 of \cite{Koj} either $T$ is the fiber of some fibration or there is a finite cover $\bar{M} \stackrel{p} \rightarrow M$ with large $b_1$, say $b_1(\bar{M}) \geq 4$. We assume that the latter holds. Then the pullback $\bar{E} = p^*E$ will be symplectic with canonical class $\bar{K} = p^*K$, symplectic form $\bar{\omega} = p^* \omega$ and $b_2^+(\bar{E}) \geq 2$. Then for any $Spin^c$-structure $\xi_{can} \otimes F$ that has nontrivial SW invariant we have by \cite{Tau2}
\[ 0 \leq F. [\bar{\omega}] \leq \bar{K}.[\bar{\omega}]. \]
Moreover, since $\bar{K}$ is torsion and equality on the left implies $F = 0$, we conclude that in fact $\bar{K} =0$. Thus $\bar{K}=0$, so $\bar{\xi}_{can}$ is trivial and this is the only $Spin^c$-structure with nonzero SW invariant. We now need to consider two cases. We first assume that $e(E)$ and hence $e(\bar{E})$ is nontorsion. In this case we compute
\[ \pm 1 = \sum_ {\xi^* \in Spin^c(\bar{E})} SW^4_{\bar{E}}(\xi^*) = \sum_ {\xi^* \in Spin^c(\bar{E})} \sum_{\xi^* \equiv \  \xi \ mod \ \bar{e}} SW^3_{\bar{M}}(\xi) = \sum_ {\xi \in Spin^c(\bar{E})} SW^3_{\bar{M}}(\xi), \]
where the second inequality follows from Theorem 1 in \cite{Bal}. However as $b_1(\bar{M}) \geq 4$ this sum is zero (cf. \cite{Tur} p.114) a contradiction. If the Euler class is torsion we may assume by Proposition \ref{killtorsion} that it is indeed zero and the above calculation reduces to
\[ \pm 1 = \sum_ {\xi \in Spin^c(\bar{E})} SW^4_{\bar{E}}(\xi) = \sum_ {\xi \in Spin^c(\bar{E})} SW^3_{\bar{M}}(\xi) = 0. \]
In either case we obtain a contradiction and hence $M$ must fiber over $S^1$.
\end{proof}

As a consequence of this theorem we conclude that Conjecture \ref{strong} holds if $M$ has vanishing Thurston norm or is Seifert fibered.
\begin{cor}\label{Seifertcase}
Conjecture \ref{strong} holds if $M$ is Seifert fibered or $||\ ||_T \equiv 0$.
\end{cor}
\begin{proof}
If $M$ has vanishing Thurston norm, then by Theorem \ref{fiberedfaces} we conclude that if one class in $H^1(M)$ can be represented by a fibration then so can all classes and by the construction of \cite{FGM} every bundle over $M$ admits an $S^1$-invariant symplectic form. If $M$ is Seifert fibered it either has vanishing Thurston norm by Proposition \ref{SFnorm} and we proceed as in the previous case or $M$ has a horizontal surface and the claim follows by Example \ref{SFfiberedfaces} above.
\end{proof}
\section{The case where $E$ admits a Lefschetz fibration}\label{The case where E admits a Lefschetz fibration}
In \cite{CM} Chen and Matveyev showed that if $S^1 \times M$ admits a symplectic Lefschetz fibration then $M$ fibers over $S^1$. This was extended by Etg\"u in \cite{Etg} to the case where the fibration may or may not be symplectic. In this section we shall show that the same statement holds for arbitrary $S^1$-bundles. Let us begin with some definitions and basic facts concerning Lefschetz fibrations. 

\begin{defn}
Let $E$ be a compact, connected, oriented smooth 4-manifold, a Lefschetz fibration is a map $E \stackrel{p} \rightarrow B$ to an orientable surface so that any critical point has a chart on which $p(z_1,z_2) = z_1^2 + z_2^2$.
\end{defn}

We list some basic properties of Lefschetz fibrations (for proofs see \cite{GS}).
\begin{enumerate}
\item There are finitely many critical points, so the generic preimage of a point will be a surface and we may assume that this is connected. To each critical point one associates a vanishing cycle in the fiber.
\item A Lefschetz fibration admits a symplectic form so that the fiber is a symplectic submanifold if the class $[F]$ of the fiber is nontorsion in $H_2(E)$. Moreover this is always true if $\chi(F) \neq 0$.
\item We have a formula for the Euler characteristic given by
\[ \chi(E) = \chi(B).\chi(F) + \# \{ critical \ points \}. \]
\end{enumerate}

We will first show that for a symplectic circle bundle any Lefschetz fibration will actually be a proper fibration, i.e. cannot have any critical points. The following lemma is essentially Lemma 3.4 of \cite{CM}.
\begin{lem}\label{novanishingcycles}
Let $S^1 \to E \stackrel{\pi} \rightarrow M$ be a circle bundle that admits a Lefschetz fibration $E \stackrel{p} \rightarrow B$, then $p$ has no critical points.
\end{lem}
\begin{proof} We first consider the case where $F = S^2$, then since $E$ is spin and hence has an even intersection form, all critical points are non-separating in $F$. Thus we cannot have any. If $F = T^2$, the equation 
\[ 0 = \chi(E) = \chi(B).\chi(F) + \# \{  critical \ points \}. \]
implies that $E$ has no  critical points.

We now consider the case when $F$ has genus $greater$ than 1. We know that $E$ admits a symplectic Lefschetz fibration. Thus by the adjunction formula for symplectic surfaces we see that
\begin{equation} \label{eqadj} K.F = \chi_{-}(F) \neq 0 \end{equation}
where $K$ is the canonical class on $E$. If $b_2^+ > 1$ then it follows from Taubes' result that $K$ is a basic class and thus the adjunction inequality holds. In the case where $b^+(E) = 1$ we may apply the adjunction inequality exactly as in the case of $b_2^+ > 1$ by (\cite{LL} Theorem E). 
Now we assume that our fibration has a critical point and hence a vanishing cycle $\gamma$, then we know that this is nonseparating so the fiber $F$ is homologous to a surface obtained by collapsing $\gamma$ to a point and this can in turn be thought of as the image of a map $F' \stackrel{f} \rightarrow E$ where $\chi_{-}(F') < \chi_{-}(F)$. Hence the image $\pi_*[F]$ may be represented by a surface of complexity at most $\chi_{-}(F')$ (see \cite{Gab}). We know that any basic class of a circle bundle is a pullback of a class on the base (see \cite{Bal}) thus by the adjunction inequality (which still holds for $b_2^+ = 1$) and equation (\ref{eqadj})
\[ \chi_{-}(F) = |K.F| = |K.\pi_*F| \leq ||\pi_*F ||_T \leq  \chi_{-}(F') < \chi_{-}(F) \]
which is a contradiction.
\end{proof}

Our proof of Theorem \ref{Lefschetzbundles} below, which differs from those of \cite{CM} and \cite{Etg}, will rely on a theorem of Stallings that characterises fibered 3-manifolds in terms of their fundamental group. 

\begin{thm}[Stallings]\label{Stallings}
Let $M$ be a compact, irreducible 3-manifold and suppose there is an extension\[ 1 \to G \to \pi_1(M) \to \mathbb{Z} \to 1\]
where $G$ is finitely generated and $G \neq \mathbb{Z}_2$, then $M$ fibers over $S^1$.
\end{thm}

We now come to the main result of this section.
\begin{thm}\label{Lefschetzbundles}
Let $E \stackrel{\pi} \rightarrow M$ be a symplectic circle bundle over an irreducible base $M$. If $E$ admits a Lefschetz fibration, then $M$ fibers over $S^1$.
\end{thm}
\begin{proof}
First of all by Lemma \ref{novanishingcycles} we have that $E$ actually admits a fibration $F \to E \stackrel{p} \rightarrow B$. In addition we note that the fiber $\gamma$ of any circle bundle lies in the centre of its total space, $\pi_1(E)$. We shall have to consider two distinct cases according to whether $\gamma$ is in the kernel of $p_*$ or not.\\

\emph{Case 1}: $p_* (\gamma) \neq 1$.\\

\noindent Since $\gamma$ was central in the fundamental group of $E$ the fact that $p_*(\gamma)$ is nontrivial in $\pi_1(B)$ means that $B$ must be a torus. Hence the long exact homotopy sequence of the fibration gives the following short exact sequence
\begin{equation}\label{seq_1}   1 \to \pi_1(F) \to \pi_1(E) \stackrel{p_*} \rightarrow \pi_1(T^2) = \mathbb{Z}^2 \to 1. \end{equation}
Since $M$ is assumed to be irreducible and hence aspherical we also have the following exact sequence from the homotopy exact sequence of the fibration $S^1 \to E \stackrel{\pi} \rightarrow M$:
\begin{equation}\label{seq_2}   1 \to \pi_1(S^1) = \langle \gamma \rangle \to \pi_1(E) \stackrel{\pi_*} \rightarrow \pi_1(M) \to 1. \end{equation}
Because $\gamma$ is central in $\pi_1(E)$, the sequence (\ref{seq_1}) gives the following exact sequence
\[ 1 \to \pi_1(F) \to \pi_1(E)/ \langle \gamma \rangle \stackrel{p_*} \rightarrow \mathbb{Z}^2/ \langle p_*\gamma \rangle \to 1. \]
Moreover since $p_*\gamma \neq 1$ we have that $\mathbb{Z}^2/ \langle p_*\gamma \rangle = \mathbb{Z} \oplus \mathbb{Z}_k$ for some $k$. If we let $H = p_*^{-1}(\mathbb{Z}_k)$ we see that $H$ has $\pi_1(F)$ as a finite index subgroup and is thus also finitely generated. Then by taking the projection to $\mathbb{Z}$ in the above sequence we obtain
\[ 1 \to H \to \pi_1(E)/ \langle \gamma \rangle  = \pi_1(M) \stackrel{p_*} \rightarrow \mathbb{Z} \to 1. \]
This is exact and $H \neq \mathbb{Z}_2$ since it contains $\langle \gamma \rangle$. As $M$ is irreducible, the hypotheses of Theorem \ref{Stallings} are satisfied and we conclude that $M$ fibers over $S^1$.\\

\emph{Case 2}: $p_* (\gamma) = 1$.\\

\noindent In this case $\langle \gamma \rangle \subset \pi_1(F)$ and hence $F = T^2$. Thus sequence (\ref{seq_1}) above yields the following
\[ 1 \to \mathbb{Z}^2 \to \pi_1(E) \stackrel{p_*} \rightarrow \pi_1(B) \to 1\]
and $\langle \gamma \rangle \subset \mathbb{Z}^2$. Again by taking the quotient by $\langle \gamma \rangle$ we obtain the following short exact sequence
\[ 1 \to \mathbb{Z} \oplus \mathbb{Z}_k = \mathbb{Z}^2/ \langle \gamma \rangle\to \pi_1(E)/\langle \gamma \rangle = \pi_1(M) \stackrel{p_*} \rightarrow \pi_1(B) \to 1.\]
However since $M$ is irreducible and hence prime and $\pi_1(M)$ is infinite it follows from (\cite{Hem}, Corollary 9.9) that $\pi_1(M)$ is torsion free. Hence $k = 0$ and $\pi_1(M)$ contains an infinite cyclic normal subgroup, thus by (\cite{Hem}, Corollary 12.8) it is in fact Seifert fibered and the result follows from Corollary \ref{Seifertcase} above. 
\end{proof}

Theorem \ref{Lefschetzbundles} then allows us to prove Conjecture \ref{strong} under the assumption that the total space is a complex manifold.
\begin{cor}
Conjecture \ref{strong} holds in the case that $E$ is a complex manifold.
\end{cor}
\begin{proof}
By considering the Kodaira classification and noting that $E$ is spin, symplectic and has $\chi(E) =0$ one concludes that one of the following must hold (cf. \cite{Etg} Theorem 5.1)
\begin{enumerate}
\item $E= S^2 \times T^2$
\item $E$ is a $T^2$-bundle over $T^2$
\item $E$ is a Seifert fibration over a hyperbolic orbifold.
\end{enumerate}
If $E= S^2 \times T^2$ then $M = S^2 \times S^1$ and one clearly has an $S^1$-invariant symplectic form. In the second case it follows from the argument above that $M$ is a $T^2$-bundle over $S^1$ and hence has vanishing Thurston norm. In the final case $M$ must be Seifert fibered as in Case 2 in the proof of Theorem \ref{Lefschetzbundles} and hence the claim holds in the latter two cases by Corollary \ref{Seifertcase}.
\end{proof}

 \end{document}